\documentclass[12pt]{article}

\usepackage[utf8]{inputenc}

\usepackage[english]{babel}

\usepackage{anysize}
\marginsize{2cm}{2cm}{1.5cm}{1.5cm}

\usepackage{amssymb}
\usepackage{latexsym}
\usepackage{mathrsfs}

\usepackage{amsthm}
\usepackage{amsmath}

\usepackage{indentfirst}

\usepackage{color}

\usepackage{ifpdf}

\ifpdf
\usepackage[pdftex]{graphicx}
\else
\usepackage{graphicx}
\fi

\DeclareMathOperator{\supp}{supp}

\theoremstyle{plain}
\newtheorem{theorem}{Theorem}

\newtheorem{corollary}{Corollary}

\theoremstyle{definition}
\newtheorem{definition}{Definition}	
\newtheorem{remark}{Remark}	

\author{Arseniy V. Akopyan\thanks{The author was supported by the Dynasty foundation and the President's of Russian Federation grant MK-3138.2014.1.} \and 
Sergey A. Pirogov\thanks{The author was supported by the Russian Foundation for Basic Research grant 13-01-12410} \and
Aleksandr N. Rybko\thanks{The author was supported by the Russian Foundation for Basic Research grant 14-01-00319}
}

\title{Invariant measures of genetic recombination process}

\begin{document}
	\maketitle
	
\begin{abstract}
	We construct the non-linear Markov process connected with biological model of bacterial genome recombination.
	The description of invariant measures of this process gives us the solution of one problem in elementary probability theory.
\end{abstract}

The genetic recombination in bacteria can be formally described in the following way~\cite{pirogov2013recombination}.
Let $\Lambda=\{1, 2, \dots, n\}$ be a finite set and for any $i \in \Lambda$ we have a finite alphabet $K_i$.
We call by genomes elements of the set $X =\prod_{i \in \Lambda} K_i$, i.~e. words in the alphabet depending on $i$.
Suppose we have a set $\mathcal{J}$ of subsets $I \subset \Lambda$.
These subsets we call \emph{frames}.
A system of frames is called the $T_0$-system if for any $i \ne j \in \Lambda$ there is a~frame $I\in \mathcal{J}$ for which either $i \in I$, $j \notin I$ or $j \in I$, $i \notin I$.

This property we named $T_0$ by the analogue with Kolmogorov's $T_0$-axiom in the general topology.

The restriction of a word $x = \{x_i, i \in \Lambda\}$ on a subset $M \subset \Lambda$ we denote by $x_M$.

\begin{definition}
	\label{def:i-recombination}
	The transform of a word $x=(x_I,x_{\Lambda\setminus I})$ to the word $\tilde x=(y_I, x_{\Lambda\setminus I})$ is called the \emph{$I$-recombination} of the word $x$ with the word $y$.
\end{definition}

Let us suppose that for any $I \in \mathcal{J}$ we have a similarity function $\phi_I(x_I, y_I)$ which we suppose to be symmetric ($\phi_I(x_I, y_I)=\phi_I(y_I, x_I)$) and strictly positive. 
 For given $I$ we consider the symmetric matrix $\Phi_I=(\phi_I(x_I, y_I))$. 
The set of matrices $\mathcal{R}=\{\Phi_I, I \in \mathcal{J}\}$ we call the \emph{legend} of recombination.

Suppose we have a probability measure $\mu$ on the space $X$.
A non-linear Markov process \cite{mckean1966class} of recombination is defined by its transition rates.
By definition, for each $I \in \mathcal{J}$ the transition rate $\lambda_I(x, \tilde x, \mu)$ of the word $x=(x_I,x_{\Lambda\setminus I})$ to the word $\tilde x=(y_I, x_{\Lambda\setminus I})$ equals $\phi_I(x_I, y_I)\mu_I(y_I)$.
(Here and below we denote by $\mu_I$ the corresponding marginal distribution, i.~e. the projection of the measure $\mu$).

We suppose that we have the initial measure $\mu^0$, i.~e.~the distribution of the word~$x(0)$, and for $t>0$ transition rates $\lambda_I(x, \tilde x, \mu^t)$ are defined by the measure $\mu^t$, which is the distribution of the word $x(t)$.
So, the distribution $\mu^t$ satisfies the nonlinear differential equations:
\begin{equation}
	\label{eq:diff equation}
	\frac{d \mu(x)}{dt}=\sum_I \sum_{y_I}\left(\phi_I(y_I, x_I)\mu_I(x_I)\mu(x_{\Lambda\setminus I}, y_I) - \phi_I(x_I, y_I) \mu_I (y_I) \mu(x)\right).
\end{equation}


For a given legend $\mathcal{R}$ let us define the following properties of the measure $\mu$.
\begin{definition}
	\label{def:r-stable measure}
	 Probability measure $\mu$ is called\\
	 a) $\mathcal{R}$-stable, if it is a fixed point for equation \eqref{eq:diff equation}.\\
	 b) $\mathcal{J}$-separated, if for any $I \in \mathcal{J}$ two sets of random variables $x_I$ and $x_{\Lambda \setminus I}$ are independent with respect to measure $\mu$.
\end{definition}

\begin{theorem}
	\label{thm:r-stable measure if T_o-sistem}
	The measure $\mu$ is $\mathcal{R}$-stable if and only if it is $\mathcal{J}$-separated.
\end{theorem}

	We supposed that $\mathcal{J}$ is $T_0$-system (for Theorem~\ref{thm:r-stable measure if T_o-sistem} this condition can be omited, see Remark~\ref{rem:not T-0 system}).

The ``if'' part of Theorem \ref{thm:r-stable measure if T_o-sistem} is trivial:
any $\mathcal{J}$-separated measure $\mu$ is $\mathcal{R}$-stable.
Indeed, if the measure $\mu$ is $\mathcal{J}$-separated, then $\mu(x_{\Lambda\setminus I}, y_I)=\mu_{\Lambda \setminus I}(x_{\Lambda\setminus I})\mu(y_I)$, and $\mu(x)=\mu_{\Lambda\setminus I}(x_{\Lambda\setminus I})\mu(x_I)$.
Therefore, by the symmetry of the function $\phi_I$, all the summands in the r.~h.~s. of differential equation \eqref{eq:diff equation} are $0$.


Now we derive Theorem~\ref{thm:r-stable measure if T_o-sistem} in direction ``only if'' from the stronger theorem.

\begin{theorem}
	\label{thm:mu0 goes to bernoulli}
	Let $\mu^0$ be an arbitrary probability measure on $X$.
	Then the trajectory $\mu^t$ in the space of measure (the solution of differential equation~\eqref{eq:diff equation}with the initial condition~$\mu^0$) for $t \rightarrow \infty$ tends to the set to of the $\mathcal{J}$-separated measures $\nu$ such that $\nu_i=\mu_i^0$.
\end{theorem}

\begin{remark}
	\label{rem:measure product}
As we will see later if $\mathcal{J}$ is $T_0$-system, then this set of $\mathcal{J}$-separated measures consist of the unique point $\nu=\prod_i \mu_i^0$.
\end{remark}

\begin{proof}[Proof of Theorem~\ref{thm:mu0 goes to bernoulli}]
The proof is based on Lyapunov method.

For the Lyapunov function we take the Shannon entropy of the measure $\mu$ (with the minus sign) $H(\mu)=\sum_x \mu(x)\ln \mu(x)$.

For a given frame $I \subset \mathcal{J}$ let us consider the differential equation containing only the summands with this $I$:
\begin{equation}
	\label{eq:modification of uravnenie}
	\frac{d \mu(x)}{dt} = \sum_{y_I}\left(\phi_I(y_I, x_I) \mu_I(x_I)\mu(x_{\Lambda\setminus I}, y_I) - \phi_I (x_I, y_I) \mu_I(y_I) \mu(x)\right).
\end{equation}
Summing this equation over $x_{\Lambda \setminus I}$ we get $\frac{d \mu_I}{dt}=0$.
So, $\mu_I$ does not depend on time, and the right hand side of~\eqref{eq:modification of uravnenie} we can consider as the direct Kolomogorov equation (i.~e. the linear differential equation for the measure) for the Markov process with constant transition rates $x=(x_I, x_{\Lambda\setminus I}) \rightarrow \tilde x=(y_I, x_{\Lambda\setminus I})$ equal to $\lambda_I(x_I, y_I)=\phi_I(x_I,y_I)\mu_I(y_I)$.

This process does not change $x_{\Lambda\setminus I}$ but it changes the mutual distribution of $x_I$ and $x_{\Lambda\setminus I}$.
For fixed $x_{\Lambda\setminus I}$ it is irreducible continuous time Markov chain on the set $\supp \mu_I$.
We suppose also that $\mu_{\Lambda\setminus I}(x_{\Lambda\setminus I})>0$.

Consider the process with fixed $x_{\Lambda\setminus I}$, the previous transition rates $\lambda_I(x_I, y_I) = \phi_I(x_I,y_I)\mu_I(y_I)$ and an arbitrary ``wrong''  initial distribution $\tilde \mu_I$ on the set $\supp \mu_I$.
It is well known \cite{feller10introduction}, that due to the irreducibility the distribution of this Markov process converges to $\mu_I$ as $t \rightarrow \infty$.

Moreover, the Kullback--Leibler divergence $H(\tilde \mu|\mu)=\sum_{x_I}\tilde \mu_I(x_I)\ln \frac{\tilde \mu_I(x_I)}{\mu_I(x_I)}$ has the strictly negative time derivative via the direct Kolmogorov equation for this process \cite{batishcheva2005, malyshev2004random, pirogov2013recombination}.
Thus if we set $\tilde \mu_I = \frac{\mu(x_I, x_{\Lambda \setminus I})}{\mu_{\Lambda \setminus I}(x_{\Lambda \setminus I})}$, then for the evolution governed by equation~\eqref{eq:modification of uravnenie} the (minus-)entropy
\begin{multline}
H(\mu)=\sum_{x_I, x_{\Lambda\setminus I}} \mu(x_I, x_{\Lambda \setminus I})\ln \mu(x_I, x_{\Lambda \setminus I})=\\
=
\sum_{x_I, x_{\Lambda\setminus I}} \mu_{\Lambda \setminus I}(x_{\Lambda \setminus I})\frac{\mu(x_I, x_{\Lambda \setminus I})}{\mu_{\Lambda \setminus I}(x_{\Lambda \setminus I})} \ln \frac{\mu(x_I, x_{\Lambda \setminus I})}{\mu_{\Lambda \setminus I}(x_{\Lambda \setminus I})\mu_I(x_I)} +\\
+\sum_{x_I}\mu_I(x_I)\ln \mu_I(x_I)+\sum_{x_{\Lambda \setminus I}}\mu_{\Lambda \setminus I}(x_{\Lambda \setminus I})\ln \mu_{\Lambda \setminus I} (x_{\Lambda \setminus I})
\end{multline}
has a strictly negative time derivative if $\mu(x_I, x_{\Lambda \setminus I})$ does not coincide with $\mu_I(x_I)\mu_{\Lambda \setminus I}(x_{\Lambda \setminus I})$.
(Here $0\ln 0=0$.)
For the complete equation~\eqref{eq:diff equation} the (minus-)Shannon entropy has a strictly negative time-derivative if the equality $\mu(x_I, x_{\Lambda \setminus I})=\mu_I(x_I)\mu_{\Lambda \setminus I}(x_{\Lambda \setminus I})$ is violated for some $I \in \mathcal{J}$.

So, any trajectory of equation~\eqref{eq:diff equation} converges to the set of $\mathcal{J}$-separated measures as $t\rightarrow \infty$.

Moreover, if $i \in I$, then $\mu_i$ does not change via equation~\eqref{eq:modification of uravnenie}, since $\mu_i$ is the projection of $\mu_I$.
The same holds for $i \notin I$ because $\mu_i$ is the projection of $\mu_{\Lambda\setminus I}$.
Therefore $\mu_i$ is not time depended for any solution~\eqref{eq:diff equation}.
Thus, any limit point of a solution of this equation is a $\mathcal{J}$-separated measure $\nu$ with marginal distributions $\nu_i=\mu_i^0$.
\end{proof}

Let us describe this set of measures $\nu$ and prove the statement of Remark~\ref{rem:measure product}.
That follows from the next Theorem.

\begin{theorem}
	\label{thm:measure is bernoulli}
	Let $\xi_i, i \in \Lambda$ be the random variables. 
	Suppose for any $i, j$, there is a subset $I \subset \Lambda$ containing exactly one element from $\{i, j\}$ and such that $\xi_I=\{\xi_i, i \in I\}$ and $\xi_{\Lambda \setminus I}=\{\xi_i, i \in \Lambda \setminus I\}$ are independent.
	Then all variables $\xi_i, i \in \Lambda$ are independent.
\end{theorem}

\begin{proof}
 The collection of subsets $I$ (frames) is $T_0$-system in the sense of the definition above.
 For any $\Lambda$ and $M\subset \Lambda$ by $J_M$ we denote the set of frames $I \cap M$ on the
 set $M$ for $I \in J$.
 It is evident that $J_M$ is a $T_0$-system if $J$ is a $T_0$-system. 
 Now we use the induction on the cardinality of $\Lambda$.
 For $|\Lambda|=1$ the statement is trivial. 
 Let $|\Lambda|>1$.
 Fix some frame $M \in J$ which is a proper subset of $\Lambda$.
 Then $J_M$ and $J_{\Lambda\setminus M}$ are $T_0$-systems.
 Cardinalities of $M$ and $\Lambda \setminus M$ are less than the cardinality of $\Lambda$.
 So we can apply the induction hypothesis and conclude that the random variables $x_i$, $i \in M$, are 
 mutually independent and random variables $x_i$, $i \in \Lambda\setminus M$, are mutually independent.
 But these two sets of random variables are independent by the condition of the theorem.
 So all random variables $x_i$, $i \in \Lambda$, are mutually independent.
 By the induction the theorem follows.
\end{proof}

An example of an application of this Theorem following statement.
\begin{corollary}
	\label{cor: matrix}
	If there is a random matrix $\xi_{i,j}$ such that its columns are independent and its rows are independent.
	Then all elements $\xi_{i,j}$ are independent.
\end{corollary}

The description of limit measure $\nu$ implies the following statement. 
\begin{corollary}
	\label{cor:combinatorial refolmulation}
	If $\mathcal{Y}=\{x^1, \dots, x^m\}$ is a set of words such that all letters $x_i^j$ for given $i$ take all values from $K_i$, then any word $x \in X$ can be obtained from words of the set $\mathcal{Y}$ by the finite sequence of recombinations.
\end{corollary}

Informally, from the set of genoms having all possible letters on any place we can obtain any genom by recombinations.

\begin{proof}
	Conisder the non-linear Markov recombination process with the inital measure $\mu^0$ such that $\supp \mu^0 = \mathcal{Y}$.
	Then as $t$ tends to infinity the measure $\mu^t$ tends to the product measure with strictly positive marginals.
	So the measure $\mu^t$ is strictly positive for sufficiently large $t$.
	Thus just means that any word can be obtained from the words of the set $\mathcal{Y}$ by the some finite sequence of recombination.
\end{proof}

This Corollary can also be proved without probability arguments using the induction on the cardinality as in the proof of Theorem~\ref{thm:measure is bernoulli}.

From Corollary~\ref{cor:combinatorial refolmulation} it easily follows the next statement.

\begin{corollary}
	\label{cor:R fill the space modification}
	If $\mathcal{J}$ is $T_0$-system and all the projections~$\mu_i^0$ are strictly positive then $\supp \mu^t = X$.
\end{corollary}


\begin{remark}
	\label{rem:not T-0 system}
	Let the system $\mathcal{J}$ be not a $T_0$-system.
	We say that two points $i, j \in \Lambda$ are equivalent if for any $I \in \mathcal{J}$ either $i\in I$, $j\in I$ or $i \notin I$, $j\notin I$.

Then the set $\Lambda$ is divided on equivalence classes $\Lambda=\cup \Lambda_j$, $j \in \tilde \Lambda$, where by $\tilde \Lambda$ we denote the set of equivalence classes.
The system $\mathcal{J}$ defines the system of frames $\mathcal{J}'$ on~$\tilde \Lambda$, which comes to be $T_0$-system.

Now $X=\prod_j \tilde K_j$ where $\tilde K_j = \prod_{i \in \Lambda_j}K_i$.
This construction evidently reduces the proof of Theorem~\ref{thm:r-stable measure if T_o-sistem} for arbitrary system $\mathcal{J}$ to the case of $T_0$-system.
It easy to formulate Theorem~\ref{thm:mu0 goes to bernoulli} for this case.
The invariant measure $\nu$ which is a limit of the solution of \eqref{eq:diff equation} is $\prod_{j \in \tilde \Lambda} \mu_j^0$, where~$\mu_j^0$ 
is marginal distribution of~$\mu^0$ on the set $\Lambda_j$.
\end{remark}

\section*{Aknowledgement}
	The authors are grateful to M.~S.~Gelfand and A.~S.~Kalinina for fruitful discussions. We also thanks G.~A.~Kabatiansky for useful advice and F.~V.~Petrov for the idea of the new proof of Theorem~\ref{thm:measure is bernoulli}.

\vskip 1cm
Institute for Information Transmission Problems RAS, Bolshoy Karetny per. 19, Moscow, Russia 127994
	
\end{document}